\documentclass[12pt]{amsart}

\usepackage{amsmath,amssymb}
\usepackage{amsfonts}
\usepackage{amsthm}
\usepackage{latexsym}
\usepackage{graphicx}
\usepackage{epsfig}

\newtheorem{theorem}{Theorem}[section]
\newtheorem{conj}{Conjecture}[section]
\newtheorem{lemma}{Lemma}[section]
\newtheorem{remark}{Remark}[section]
\newtheorem{coro}{Corollary}[section]

\newtheorem{prop}{Proposition}[section]

\def\RR{{\mathrm R}}
\def\WW{{\mathrm W}}
\def\EE{{\mathrm E}}
\def\Rc{{\mathrm {Rc}}}

\def\SS{{\mathrm S}}

\def\diag{\mathrm{Diag}}
\def\id{\mathrm{Id}}

\begin{document}

\title{Einstein four-manifolds of Pinched Sectional Curvature}

\author{Xiaodong Cao}

\address{Department of Mathematics,
 Cornell University, Ithaca, NY 14853-4201}
\email{cao@math.cornell.edu}

\author{Hung Tran}

\address{Department of Mathematics and Statistics,
Texas Tech University, Lubbock, TX 79409}
\email{hung.tran@ttu.edu}

\renewcommand{\subjclassname}{%
  \textup{2010} Mathematics Subject Classification}

\date{ \today}
\begin{abstract} In this paper, we obtain classification of four-dimensional Einstein manifolds with positive Ricci curvature and pinched sectional curvature. In particular, the first result deals with an upper bound on the sectional curvature, improving a theorem of E. Costa. The second is a generalization of D. Yang's result assuming an upper bound on the difference between sectional curvatures.  
\end{abstract}
\maketitle

\section{Introduction}
A Riemannian manifold $(M,g)$ is called Einstein if it satisfies 
\begin{equation}
\label{einstein}
\Rc=\lambda g,
\end{equation}
where $\Rc$ is its Ricci curvature and $\lambda$ is a constant. A fundamental problem in differential geometry is to determine whether a smooth manifold admits an Einstein metric. In dimension four, the Hitchin-Thorpe inequality gives a topological obstruction to the existence of such a structure. Currently, there is no known topological obstruction in higher dimensions. As a consequence, it is of great interests to identify all Einstein manifolds under natural and suitable assumptions. 
 
Generally, by Myer's theorem \cite{myer41}, if $\lambda>0$ then $M$ is compact and has a finite fundamental group. In dimension four, there are not many compact examples. The well-known ones are symmetric spaces including: the sphere $\mathbb{S}^4$ with round metric $g_0$, the complex protective space $\mathbb{CP}^2$ with Fubini-Study metric $g_{FS}$, the product of two unit round spheres $\mathbb{S}^2\times \mathbb{S}^2$, and their quotients. Interestingly, they all have non-negative sectional curvatures and are only known examples with that property. That motivates the following folklore  conjecture (see, for example, \cite{yangdg00}):

\begin{conj} An Einstein four manifold with $\lambda>0$ and non-negative sectional curvature must be either $(\mathbb{S}^4,g_0)$, $(\mathbb{CP}^2, g_{FS})$, $\mathbb{S}^2\times \mathbb{S}^2$ or quotient. 
\end{conj}

While the conjecture is still open, there have been various contributions from many mathematicians; see \cite{berger61, gl99, tachibana, yangdg00, costa04, brendle10einstein, cao14einstein} and the references therein. For convenience, one can normalize the metric such that $\Rc=g$. M. Berger first obtained classification under the condition that the sectional curvature is $1/4$-pinched \cite{berger61}. Recently, one interesting approach is to assume a positive lower bound on the sectional curvature, i.e., $K_{min}\geq \epsilon>0$;  this immediately implies that $K_{max}\leq 1-2\epsilon$, so the sectional curvature is $(\frac{\epsilon}{1-2\epsilon})$-pinched. That is, Berger's result implies classification for Einstein manifolds provided $K_{min}\geq \frac{1}{6}$. There have been various results in this direction. The lower bound has been improved from $\frac{1}{120}(\sqrt{1249}-23)$ by D. Yang \cite{yangdg00} to $\frac16 (2-\sqrt{2})$ by E. Costa \cite{costa04} to $\frac{1}{12}$ by the first author and P. Wu \cite{cao14einstein}. There are also classification results under related conditions such as $K\geq 0$ and positive intersection form \cite{gl99}, non-negative curvature operator \cite{tachibana}, nonnegative isotropic curvature \cite{brendle10einstein}, and 3-nonnegative curvature operator \cite{cao14einstein}. \\

As mentioned above, $K\geq 0$ implies $K\leq 1$. Thus, it is of great interests to consider upper bounds on the sectional curvature (say $K\leq 1$) instead of lower bounds. Costa \cite{costa04} observed that if $K\leq \frac{2}{3}$ then the Einstein manifold must be $(\mathbb{S}^4,g_0)$, $(\mathbb{CP}^2, g_{FS})$, or their quotients. Note that without assuming non-negativity of sectional curvature, an upper bound on $K$ gives, due to the algebraic structure, \textit{a priori} lower bound on $K$. For instance, $K\leq 1$ implies $K\geq -2$.  But a careful analysis of the differential structure would give a better bound. In particular, using \cite[Proposition 2.4(3)]{cao14einstein},  one can show that the condition  $K< 1$ is equivalent to the Riemannian curvature operator being 4-positive; hence, it follows that  \[K> \frac{1}{28}(7-\sqrt{105})\approx -.11596.\] 
That's our motivation to derive the following improvements.
\begin{theorem}
Let $(M,g)$ be a complete smooth 4-manifold such that $\Rc=g$. Assume the sectional curvature satisfies either one of the following conditions:
\begin{itemize}
\label{maineinstein}
\item[a.]$K\leq \frac{14-\sqrt{19}}{12}\approx .8034.$
\item[b.] At each point $q\in M$, for every orthonormal basis $\{e_i\}_{i=1}^4 \in T_qM$ which satisfies $K(e_1,e_2)\geq K(e_1,e_3)$, 
\[2K(e_1,e_2)+K(e_1,e_3)\geq \frac{\sqrt{19}-3}{4}\approx .3397247.\] 
\end{itemize}
Then $(M,g)$ is isometric to either $(S^4,g_{0})$, $(\mathbb{RP}^4, g_0)$, or $(\mathbb{CP}^2,\ g_{FS})$ up to rescaling.
\end{theorem}

\begin{remark} Condition (a) is an improvement of Costa's result. Condition (b) generalizes a result of \cite{yangdg00} which assumes a lower bound of $9/14\approx .642857$. As explained in Section \ref{classify}, this is essentially an upper bound on the difference between sectional curvatures. 
\end{remark}
Here is a sketch of the proof. The main idea is to make use of elliptic equations, which arise from Ricci flow computation, to study a static metric (see \cite{brendle10einstein, caotran1, cao14einstein, tran16weyl} for similar exploitation of this approach). In particular, Brendle first observed a Bochner formula for the Riemannian curvature on an Einstein manifold. Considering that equation at point $p$, which realizes the minimal sectional curvature, yields an inequality involving only zero order terms. Thus, either condition leads to a lower bound of $K$, which improves the \textit{a priori} bound. The rest of the proof is an adaptation of arguments in \cite{yangdg00, cao14einstein}: integrating a Bochner-Weitzenb\"ock identity of the Weyl tensor and applying the pinched condition imply the manifold is half-conformally flat. Then Hitchin's classification theorem of such manifolds completes our proof.  \\

The organization of the paper is as follows. The next section collects preliminaries including Berger's curvature decomposition, the inequality at point $p$ realizing the minimal sectional curvature,  and a classification by a technical condition on the Weyl tensor. Section \ref{sectionestimate} derives estimates from the algebraic structure of the curvature and inequalities resulted from the differential structure. Finally, the last section provides a proof of the main theorem. \\

{\bf Acknowledgment.} The authors would like to thank the anonymous referees for their valuable comments and suggestions to improve an earlier version of this paper. Cao's research was partially
supported by a grant from the Simons Foundation (\#280161).

\section{Preliminaries}
\subsection{\textbf{Curvature decomposition for four-manifolds}} In this subsection, we recall the curvature decomposition and several properties of an Einstein 4-manifold. 

First, on an oriented Riemannian manifold $(M,\ g)$, let $\RR, K, \Rc, \SS, \WW$ denote the Riemann curvature, sectional curvature, Ricci curvature, scalar curvature and Weyl curvature, respectively. Dimension four is special. The Hodge star operator induces a natural decomposition of the vector bundle of 2-forms,
$\wedge^2 TM$ 
\begin{equation*}
\wedge^2 TM =\wedge^+ M \oplus \wedge^- M,
\end{equation*}
where $\wedge^{\pm} M$ are the eigenspaces of eigenvalues $\pm 1$, respectively.
Elements of $\wedge^+ M$ and $\wedge^- M$ are called self-dual and
anti-self-dual 2-forms. 
Furthermore, it leads a decomposition for the curvature operator $\RR: \wedge^2 TM
\rightarrow \wedge^2 TM$:
\begin{equation*}
\RR =\left( \begin{array}{cc}
\frac{\SS}{12}\id+\WW^+ & \Rc-\frac{\SS}{4}\id \\
\Rc-\frac{\SS}{4}\id & \frac{\SS}{12}\id+\WW^-
\end{array} \right),
\end{equation*}
for $\WW^{\pm}$ the restriction of $\WW$ to $\wedge^{\pm}M$. Here $\EE=\Rc-\frac{\SS}{4}\id$ is the traceless Ricci part. 

It follows immediately that if the metric is Einstein then $\EE=0$ and  
\begin{equation}
\label{dual}
\RR =\left( \begin{array}{cc}
\frac{\SS}{12}\id+\WW^+ & 0 \\
0 & \frac{\SS}{12}\id+\WW^-
\end{array} \right):=\left( \begin{array}{cc}
\RR^+ & 0 \\
0 & \RR^-
\end{array} \right).
\end{equation}

The sectional curvature comes into play via Berger's decomposition \cite{berger61, st69}.

\begin{prop}
\label{berger}
Let $(M,\ g)$ be an Einstein four-manifold with $\Rc=\lambda g$. For
any $p\in M$, there exists an orthonormal basis $\{e_i\}_{1\leq i\leq
4}$ of $T_p M$, such that relative to the corresponding basis
$\{e_i\wedge e_j\}_{1\leq i<j\leq 4}$ of $\wedge^2 T_pM$,
$\RR$ takes the form
\begin{equation}
\label{abba}
\RR=\left( \begin{array}{cc}
A & B\\
B & A
\end{array}\right),
\end{equation}
where $A=\diag\{a_1,\ a_2,\ a_3\}$, $B=\diag \{b_1,\ b_2,\ b_3\}$.
Moreover, we have the followings:

\begin{enumerate}
\item  $a_1=K(e_1, e_2)=K(e_3, e_4)$ and $a_3=K(e_1,e_4)=K(e_2,e_3)$ realize the minimal and maximal sectional curvature at that point; 

\item $a_2=K(e_1, e_3)=K(e_2, e_4)$ and $a_1+a_2+a_3=\lambda$.

\item  $b_1=R_{1234},\ b_2=R_{1342},\ b_3=R_{1423}$.

\item  $|b_2-b_1|\leq a_2-a_1,\ |b_3-b_1|\leq a_3-a_1,\ |b_3-b_2|\leq
a_3-a_2$.
\end{enumerate}
\end{prop}

One can easily observe  that diagonalization of (\ref{abba}) becomes (\ref{dual}). To be precise, orthonormal bases of $\wedge^{\pm} M$ are given by, for $e_{ij}:=e_i\wedge e_j$,
\begin{align*}
\sqrt{2}(\omega_1^+, \omega_2^+, \omega_3^{+}) &=(e_{12}+e_{34}, e_{13}-e_{24}, e_{14}+e_{23}),\\
\sqrt{2}(\omega_1^-, \omega_2^-, \omega_3^{-}) &=(e_{12}-e_{34}, e_{13}+e_{24}, e_{14}-e_{23}).
\end{align*}
 As a consequence, $\WW^{\pm}$ are given by 
\begin{equation}\label{eigenvalueW}
\begin{cases}
\WW^+(\omega^+_i,\omega^+_j) = [(a_i+b_i)-\frac{s}{12}]\delta_{ij}, &\\
\WW^-(\omega^-_i,\omega^-_j) = [(a_i-b_i)-\frac{s}{12}]\delta_{ij}, &
\end{cases}
\end{equation}
and eigenvalues of the curvature operator are ordered,
\begin{equation}
\label{eigenvalue}
\begin{cases}
a_1+b_1 \leq a_2+b_2 \leq a_3+b_3, &\\
a_1-b_1 \leq a_2-b_2 \leq a_3-b_3. &
\end{cases}
\end{equation}

Next, we collect important properties of a closed Einstein 4-manifold. First, the curvature and topology of such a closed 4-manifold are connected via the Gauss-Bonnet-Chern formula of the Euler characteristic $\chi$ and Hirzebruch formula of the topology signature $\tau$ (cf. \cite{Besse} for more details):
\begin{align}
\label{Euler}
8\pi^2 \chi(M) &= \int_{M}(|\WW|^2-\frac{1}{2}|\EE|^2+\frac{S^2}{24}) dv;\\
\label{signature}
12 \pi^2 \tau(M) &=\int_{M}(|\WW^{+}|^2-|\WW^{-}|^2)dv. 
\end{align} 
If the metric is Einstein then $|\EE|=0$. Thus, a direct consequence of  (\ref{Euler}) and (\ref{signature}) is the the Hitchin-Thorpe inequality:
\begin{equation*}
|\tau(M)|\leq \frac{2}{3}\chi(M).
\end{equation*}

The duality decomposition (\ref{dual}) also implies that, $\RR,\RR^{\pm},\WW,\WW^{\pm}$ of an Einstein 4-manifold are all harmonic.
Using the harmonicity of $\WW^{\pm}$, A. Derdzi\'nski \cite{Derd83}
derived the following Weitzenb\"ock formula (also see  A. Besse
\cite[Prop. 16.73]{Besse}):
\begin{equation}
\label{weitzenbock}
\Delta|W^{\pm}|^2 = 2|\nabla W^{\pm}|^2
+\SS |W^{\pm}|^2 -36\det W^{\pm}.
\end{equation}
Furthermore, it was observed by Gursky-LeBrun \cite{gl99} and Yang
\cite{yangdg00} that $\WW^{\pm}$ satisfies the following  refined Kato inequality (proven to be
optimal by \cite{branson00, CGH00weight}),

\begin{equation}
\label{Kato}
|\nabla W^{\pm}|^2\geq \frac{5}{3}|\nabla|W^{\pm}||^2.
\end{equation}

In the following, we normalize the metric such that $\Rc=g$. The following elliptic equation is well-known, see \cite{H3} or
\cite{brendle10einstein},
\begin{equation*} \Delta
R(e_i,e_j,e_k,e_l) +2(B_{ijkl}-B_{ijlk}+B_{ikjl}-B_{iljk})
=2R_{ijkl},
\end{equation*}
where $B_{ijkl}=g^{mn}g^{pq}R_{imjp}R_{knlq}$. Berger's
curvature decomposition yields explicitly that
\begin{equation*}
\Delta R(e_1,e_2,e_1,e_2) +2(a_1^2+b_1^2+2a_2a_3+2b_2b_3) = 2a_1.
\end{equation*}

Let \textbf{$p$} be the point that realizes the minimum of the sectional curvature of $(M^4,g)$ by the tangent plane spanned by $\{e_1,e_2\}\subset T_p M$. For
any $v\in T_p M$ and geodesic $\gamma(t)$ with $\gamma(0)=p,\
\gamma'(0)=v$, let $\{e_1,\ e_2,\ e_3,\ e_4\}$ be a parallel
orthornormal frame along $\gamma(t)$, then we have
\begin{equation*}
\begin{split}
(D^2_{v,v}R)(e_1,e_2,e_1,e_2)(p)=D^2_{v,v}(R(e_1,e_2,e_1,e_2))(p)\geq0.
\end{split}
\end{equation*}
Hence it follows that $(\Delta R)(e_1,e_2,e_1,e_2)(p) \geq 0$. Thus, at $p$, the following holds
\begin{equation}
\begin{split}
a_1^2+b_1^2+2(a_2a_3+b_2b_3)&\leq a_1.
\end{split} \label{Hamilton}
\end{equation}
\subsection{\textbf{Examples}}
Here we state the curvature decomposition for some well-known Einstein 4-manifolds. The curvature of $\mathbb{S}^4$ ($\chi=2,\tau=0$) with standard metric $g_0$ is:
\begin{equation}
\RR=\left( \begin{array}{cc}
\frac{\SS}{12}\id &  \\
   & \frac{\SS}{12}\id
     \end{array} \right)
\end{equation}
$(\mathbb{RP}^4, g_0)$ is the quotient of $(\mathbb{S}^4,g_0)$ by the antipodal identification. The curvature of $\mathbb{CP}^2$ ($\chi=3,\tau=1$) with Fubini-Study metric $g_{FS}$ is:
\begin{equation}
\RR=\left( \begin{array}{cc}
\diag \{0,\ 0, \ \frac{\SS}{4}\}&  \\
   & \frac{\SS}{12}\id
     \end{array} \right).
\end{equation}
The self-dual part of Weyl tensor $\WW^+=\diag\{-\frac{\SS}{12}, \ -\frac{\SS}{12}, \ \frac{\SS}{6}\}$ and anti-self-dual part $\WW^{-}=0$.
 
The curvature of $\mathbb{S}^2\times \mathbb{S}^2$ ($\chi=4,\tau=0$) with the product metric is 
\begin{equation}
\RR=\left( \begin{array}{cc}
A &  0\\
0  & A
     \end{array} \right)
\end{equation} 
for $A=\diag\{0,\ 0,\ \frac{\SS}{4}\}$. The self-dual part and anti-self-dual part of the Weyl tensor are
$\WW^{\pm}=\diag\{-\frac{\SS}{12},\ -\frac{\SS}{12},\ \frac{\SS}{6}\}$. \\
\subsection{\textbf{Classification by estimates on $\WW^\pm$.}}
In this subsection, we collect a classification by a technical assumption. The result is implicit in \cite{cao14einstein} and a proof is provided for completeness.

\begin{prop}
\label{propWpm}
Let $(M, g)$ be an Einstein four-manifold $Rc=g$
such that,
\[|\WW^+|+|\WW^-|\leq \sqrt{\frac{3}{2}}\]
Then it is isometric to either $(S^4,g_{0})$, $(\mathbb{RP}^4, g_0)$, or $(\mathbb{CP}^2,\ g_{FS})$ up to rescaling. 
\end{prop}

\begin{proof}
If the manifold is half-conformally flat, then, by Hitchin's classification \cite[Theorem 13.30]{Besse}, the result follows. If not, we'll derive a contradiction. First, for some $\alpha>0$ to be determined later, and any $\epsilon>0$, there exists
$t=t(\alpha,\epsilon)\in \mathbb{R}^{+}$, such that
\begin{equation*}
\begin{cases} \int_M
(|W^+|^2+\epsilon)^{\frac{\alpha}{2}}-t(|W^-|^2+\epsilon)^{\frac{\alpha}{2}}
dv=0,\\
t(\alpha,\epsilon)\rightarrow t(\alpha, 0) \text{~ as~} \epsilon\rightarrow 0
\end{cases}
\end{equation*}

Applying the Weitzenb\"ock formula (\ref{weitzenbock}) and the refined Kato inequality (\ref{Kato}) yields,
\begin{align*}
& \Delta[(|W^+|^2+\epsilon)^\alpha+t^2(|W^-|^2+\epsilon)^\alpha]\\
=& \alpha(|W^+|^2+\epsilon)^{\alpha-2}[(|W^+|^2+\epsilon)(2|\nabla W^+|^2+\SS|W^+|^2-36\det W^+)+(\alpha-1)|\nabla|W^+|^2|^2]\nonumber \\
+& t^2\alpha(|W^-|^2+\epsilon)^{\alpha-2}[(|W^-|^2+\epsilon)(2|\nabla W^-|^2+\SS|W^-|^2-36\det W^-)+(\alpha-1)|\nabla|W^-|^2|^2] \nonumber\\
\geq& [(4-\frac{2}{3\alpha})|\nabla(|W^+|^2+\epsilon)^{\frac{\alpha}{2}}|^2+\alpha(|W^+|^2+\epsilon)^{\alpha-1}(\SS|W^+|^2-36\det W^+)]\nonumber \\
+&
t^2[(4-\frac{2}{3\alpha})|\nabla(|W^-|^2+\epsilon)^{\frac{\alpha}{2}}|^2+\alpha(|W^-|^2+\epsilon)^{\alpha-1}(\SS|W^-|^2-36\det W^-)]. (*)
\end{align*}
The choice of $t(\alpha, \epsilon)$ enables us to use the Poincar\'e inequality
\begin{equation*}
\begin{split}
&(4-\frac{2}{3\alpha})\int_M
(|\nabla(|W^+|^2+\epsilon)^{\frac{\alpha}{2}}|^2+t^2|\nabla(|W^-|^2+\epsilon)^{\frac{\alpha}{2}}|^2)
dv\\
&\geq (2-\frac{1}{3\alpha})\int_M|\nabla[(|W^+|^2+\epsilon)^{\frac{\alpha}{2}}-t(|W^-|^2+\epsilon)^{\frac{\alpha}{2}}]|^2 dv\\
&\geq (2-\frac{1}{3\alpha})\lambda_1 \int_M
[(|W^+|^2+\epsilon)^{\frac{\alpha}{2}}-t(|W^-|^2+\epsilon)^{\frac{\alpha}{2}}]^2dv,
\end{split}
\end{equation*}
where $\lambda_1$ is the lowest positive eigenvalue of the Laplace
operator. In our case that $\text{Ric}=g$,  we have $\lambda_1 \geq
\frac{4}{3}$ (see, for example, \cite{Lich58}).
Picking $\alpha=\frac{1}{3}$ which maximizes the value of 
$\frac{1}{\alpha}(2-\frac{1}{3\alpha})$, substituting $\SS=4$, and integrating the
inequality (*) yield
\begin{equation*}
\begin{split}
0\geq \frac13 \int_M&
\Big(4\Big[(|W^+|^2+\epsilon)^{\frac{1}{6}}-t(|W^-|^2+\epsilon)^{\frac{1}{6}}\Big]^2+t^2(|W^-|^2+\epsilon)^{-2/3}(4|W^-|^2-36\det
W^-)\\
&+(|W^+|^2+\epsilon)^{-2/3}(4|W^+|^2-36\det W^+)\Big)dv.
\end{split}
\end{equation*}
Using the algebraic inequality $36\text{det}(W^{\pm})\leq 2\sqrt{6}|W^{\pm}|^3$ and letting $\epsilon\rightarrow 0$, we get
\begin{align*}
0\geq & \int_M \Big(t^2|W^-|^{-4/3}(4 |W^-|^2-36\det
W^-)+4(|W^+|^{1/3}-t|W^-|^{1/3})^2\\
&+|W^+|^{-4/3}(4|W^+|^2-36\det W^+)\Big)dv,\\
\geq & \int_M \Big(t^2|W^-|^{2/3}(4-2\sqrt{6}|W^-|)+4(|W^+|^{1/3}-t|W^-|^{1/3})^2\\
&+|W^+|^{2/3}(4-2\sqrt{6}|W^+|)\Big)dv\\
\geq  & \int_M \Big(t^2 |W^-|^{2/3}(8-2\sqrt{6}|W^-|)-8t|W^+|^{1/3}|W^-|^{1/3}+|W^+|^{2/3}(8-2\sqrt{6}|W^+|)\Big)dv.
\end{align*}

The integrand is a quadratic function of $t$, with positive leading
coefficient and discriminant
\begin{align*}
D=& |W^+|^{2/3}|W^-|^{2/3}\Big[16-64+16\sqrt{6}(|W^+|+|W^-|)-24|W^+||W^-|\Big]\\
=& 16|W^+|^{2/3}|W^-|^{2/3}\Big[\sqrt{6}(|W^+|+|W^-|)-3-\frac{3}{2}|W^+||W^-|\Big].
\end{align*}
By our assumption, $D\leq 0$. So equality must happen at each point and so either $|\WW^+|^2$ or $|\WW^-|^2$ must vanish at each point. As  both are analytic functions, one must be vanishing everywhere. So the manifold is half-conformally flat, a contradiction. 
\end{proof}
\section{Estimates}
\label{sectionestimate}

In this section, we derive estimates from the algebraic structure of $\WW^{\pm}$ and inequality (\ref{Hamilton}). 
\subsection{Algebraic Estimates}
The technical lemma below estimates the norm of $\WW^{\pm}$ by pinching of sectional curvatures. 
\begin{lemma}
\label{k3k1}
Using the Berger's decomposition Prop. \ref{berger}, suppose
\begin{align*}
a_3-a_2 &= \frac{\delta}{2}\geq 0,\\
a_3+a_2 &= \alpha>0.
\end{align*}
Then 
\begin{align*} \delta &\leq 6\alpha-4;\\
|W^{+}|+|W^{-}| &\leq \frac{6\alpha-4+\delta}{\sqrt{6}},\\
|W^{+}|^2+|W^{-}|^2 &\leq \frac{1}{2}(12\alpha^2-16\alpha+\frac{16}{3}+\delta^2)\\
 &= 8(a_3^2-(1-a_1)(1-a_2)+\frac{1}{3}).
\end{align*}
\end{lemma}
\begin{proof} We'll prove the first estimate while the second one follows from the same method. Let $\{\lambda_i, \mu_i\}_{i=1}^3$ be eigenvalues of $\WW^\pm$. By (\ref{eigenvalueW}) the assumptions above translate to
\begin{align*}
\lambda_3+\lambda_2+\mu_3+\mu_2 =\alpha_1=2(\alpha-\frac{2}{3}),\\
\lambda_3-\lambda_2+\mu_3-\mu_2 &= \delta,\\
-\frac{\lambda_3}{2}\leq \lambda_2\leq \lambda_3,\\
-\frac{\mu_3}{2}\leq \mu_2\leq \mu_3.
\end{align*}
Then
\begin{align*} 
|W^{+}|+|W^{-}|&= \sqrt{2}(\sqrt{\lambda_3^2+\lambda_2^2+\lambda_3\lambda_2}+\sqrt{\mu_3^2+\mu_2^2+\mu_3\mu_2},\\
&=\sqrt{2}(\sqrt{\frac{3}{4}(\lambda_3+\lambda_2)^2+\frac{1}{4}((\lambda_3-\lambda_2)^2}+\sqrt{\frac{3}{4}(\mu_3+\mu_2)^2+\frac{1}{4}((\mu_3-\mu_2)^2}).
\end{align*}
So we consider the following problem. 
For,
\begin{align*}
\lambda_3+\lambda_2 &=p,\\
\lambda_3-\lambda_2 &=q,\\
\mu_3+\mu_2 &=m,\\
\mu_3-\mu_2 &=n.
\end{align*}
The constraints become:
\begin{align*}
p+m &= \alpha_1,\\
q+n &= \delta,\\
0 &\leq q\leq 3p,\\
0 &\leq n\leq 3m.
\end{align*} 
The goal is to maximize 
\[f(p,q,m,n)=\sqrt{3p^2+q^2}+\sqrt{3m^2+n^2}.\]
First, since the constraint is a closed set, the maximum exists. 



 
Next, we consider the problem with 2 variables $m, n$. The constraints become, 
\begin{align*}
0 &\leq 3m-n=\ell \leq 3\alpha_1-\delta,\\
0 &\leq n\leq \delta.
\end{align*}
The function to maximize is
\begin{align*}
f(m,n) &= \sqrt{3(\alpha_1-m)^2+(\delta-n)^2}+\sqrt{3m^2+n^2},\\
\sqrt{3}f(m,n) &=\sqrt{(3\alpha_1-\ell-n)^2+3(\delta-n)^2}+\sqrt{(\ell+n)^2+3n^2},\\
 &=g(\ell,n).
\end{align*}
We consider the following cases.

\textbf{Case 1:} $n=0$ then 
\begin{align*}
g(\ell,0) 
&=\sqrt{(\ell-3\alpha_1)^2+3\delta^2}+\ell.
\end{align*}
Since the function $f(x)=\sqrt{x^2+a^2}+x$, $a>0$, is strictly increasing, 
\[g(\ell,0)\leq 3\alpha_1+\delta.\]

\textbf{Case 2:} $\ell=0$ then 
\begin{align*}
g(n,0) &= \sqrt{4n^2-6(\alpha_1+\delta)n+ 9\alpha_1^2+3\delta^2}+2n,\\
&=\sqrt{(2n-\frac{3}{2}(\alpha_1+\delta))^2+(\frac{3\sqrt{3}}{2}\alpha_1-\frac{\sqrt{3}}{2}\delta)^2}+2n\\
& \leq 3\alpha_1+\delta. 
\end{align*}

\textbf{Case 3:} $n=\delta$ then 
\begin{align*}
g(\ell,\delta) 
&=(3\alpha_1-\delta)-\ell+\sqrt{(\ell+\delta)^2+3\delta^2},\\
&\leq 3\alpha_1+\delta.
\end{align*}

\textbf{Case 4:} $\ell=3\alpha_1-\delta$ then 
\begin{align*}
g(3\alpha_1-\delta,n) &= 2(\delta-n)+\sqrt{(2n+\frac{1}{2}(3\alpha_1-\delta))^2+\frac{3}{4}(3\alpha_1-\delta)^2}\\
&\leq 3\alpha_1+\delta.
\end{align*}

\textbf{Case 5:} At a critical point,
\begin{align*}
0=\partial_\ell g &=-\frac{3\alpha_1-\ell-n}{\sqrt{(3\alpha_1-\ell-n)^2+3(\delta-n)^2}}+\frac{\ell+n}{\sqrt{(\ell+n)^2+3n^2}},\\
0=\partial_n g &=-\frac{3\alpha_1+3\delta-\ell-4n}{\sqrt{(3\alpha_1-\ell-n)^2+3(\delta-n)^2}}+\frac{4n+\ell}{\sqrt{(\ell+n)^2+3n^2}}
\end{align*}
Therefore, at that point,
\begin{align*}
\frac{\ell+n}{(3\alpha_1-\delta-\ell)+(\delta-n)} &=\frac{4n+\ell}{(3\alpha_1-\delta-\ell)+4(\delta-n)}\\
&=\frac{n}{\delta-n}=\frac{\ell}{3\alpha_1-\delta-\ell}=\frac{1}{x}.
\end{align*}
Then, 
\begin{align*}
g(\ell,n)&=(x+1)\sqrt{(\frac{3\alpha_1-\delta}{x+1}+\frac{\delta}{x+1})^2+(\frac{\delta}{x+1})^2}\\
&=\sqrt{9\alpha_1^2+\delta^2}\leq 3\alpha_1+\delta.
\end{align*}
\end{proof}

\begin{remark} The first estimate generalizes and unifies \cite[Lemma 4.1]{yangdg00} which treats the case $\alpha\leq 1-\epsilon$, $\delta\leq 2(1-3\epsilon)$ in part (a.) and $\alpha\leq 1$, $\delta \leq 1-6\epsilon$ in part (b.).   
\end{remark}
\begin{remark}
The condition that $6\alpha+\delta-4\leq 3$ implies $3$-non-negative curvature.
\end{remark}
\begin{coro}
\label{eulerestimate}Let $(M, g)$ be an Einstein four-manifold $\Rc=g$ 
such that at each point, $\alpha\leq a_1\leq a_3\leq \beta$, then 
\[8\pi^2 \chi(M) \leq \Big(8(\beta^2-(1-\alpha)(\alpha+\beta)+\frac{10}{3}\Big)\text{Vol}(M).\]
\end{coro}
\begin{proof}
Recall, 
\[8\pi^2 \chi(M)=\int_M (|\WW^+|^2+|\WW^-|^2+\frac{\SS^2}{24}) d\mu.\]
By Lemma \ref{k3k1}, at each point,
\[|W^{+}|^2+|W^{-}|^2 \leq 8(a_3^2-(1-a_1)(1-a_2)+\frac{1}{3}).\]
So it remains to maximize, given the pinching condition on $a_1, ~a_3$,
\[f(a_1, a_3)=a_3^2-(1-a_1)(a_1+a_3).\]
By the algebraic properties of $a_1, ~a_3$, the domain here is a quadrilateral determined by lines $x=\alpha, y=\beta, 2x+y=1, 2y+x=1$ (which is already within the half-plane $y\geq x$). So standard technique yields, 
\[|W^{+}|^2+|W^{-}|^2 \leq 8(\beta^2-(1-\alpha)(\alpha+\beta)+\frac{1}{3}).\]
The result then follows. 
\end{proof}
\begin{remark}
In \cite{gl99}, the authors show that if $\WW^+\not \equiv 0$ then,
\[\int_M |\WW^+|^2 dv\geq \frac{2}{3}\text{Vol}(M).\]
That is, if the Einstein 4-manifold is not half-conformally flat then, 
\[8\pi^2 \chi(M) \geq 2 \text{Vol}(M).\]
They also observe that if $a_1\geq 0$ then, 
\[8\pi^2 \chi(M) \leq \frac{10}{3}\text{Vol}(M).\]
As a consequence, it follows that \[9\geq \chi(M)>\frac{15}{4}|\tau(M)|.\]
Therefore, there are only finitely many homeomorphism types for an Einstein 4-manifold with non-negative sectional curvature and not half-conformally flat. Corollary \ref{eulerestimate} then gives a more precise description of the relation between the topology type and bound on the sectional curvature. For instance, if $a_1> \alpha=\frac{2-\sqrt{3}}{6}\approx .04466$ then we could choose $\beta=1-2\alpha$ and then $(|\tau|, \chi)$ could only be one of the following choices $(1, 5), (1,7),$ $(0, 2), (0,4), (0,6)$. 
\end{remark}
\subsection{Differential Estimates} 
Here we derive several consequences of (\ref{Hamilton}). First we observe the following. 
\begin{lemma}
\label{algebraic2}
Let $xy\leq 0$ and assume
\begin{align*}
|2x+y| &\leq a,\\
|x-y| &\leq b.
\end{align*}
If $2a<b$ then \[4xy+x^2+y^2\geq \frac{1}{3}(2a^2-2ab-b^2).\] 
If $2a\geq b$ then  \[4xy+x^2+y^2\geq -\frac{1}{2}b^2.\]
\end{lemma}
\begin{proof}
Let $m=2x+y$ and $n=x-y$ then
\begin{align*}
9(4xy+x^2+y^2) &=(m+n)^2+(m-2n)^2+4(m+n)(m-2n)\\
&=6m^2-3n^2-6mn\:=f(m,n).
\end{align*}
Consider the region $D=\{-a\leq m\leq a; -b\leq n\leq b\}$. The only critical point of $f(m,n)$ is $(0,0)$ and $f(0,0)=0$. So we consider the function along the boundary of $D$ and the result follows. 
\end{proof}
Recall that \textbf{$p$} is the point that realizes the minimum of the sectional curvature of $(M^4,g)$ by the tangent plane spanned by $\{e_1,e_2\}$. At point $p$, we get:
\begin{equation*}
\begin{split}
a_1^2+b_1^2+2(a_2a_3+b_2b_3)&\leq a_1.
\end{split} 
\end{equation*}
Also by Prop \ref{berger} we observe,
\begin{align*}
|b_2-b_3| &\leq a_3-a_2=b,\\
|b_2-b_1|=|2b_2+b_3|&\leq a_2-a_1=a,\\
b_1^2+2b_2b_3 &=b_2^2+b_3^2+4b_2b_3.
\end{align*}
So Lemma \ref{algebraic2} implies the followings: 
\begin{itemize}
\item If $2a<b$ or $a_1+1\geq 4a_2$ then 
\[b_1^2+2b_2b_3\geq (a_2-a_1)^2 -\frac{1}{3}(a_3-a_1)^2.\]
\item If $2a\geq b$ or $a_1+1\leq 4a_2$ then 
\[b_1^2+2b_2b_3\geq -\frac{1}{2}(a_3-a_2)^2.\]
\end{itemize}
Consequently, it is possible to obtain a lower bound for sectional curvature given an upper bound.
\begin{lemma}
\label{Kupper}
At point $p$, suppose $a_3=\alpha\leq 1$ then we have:
\begin{align*}
4a_2 &\leq a_1+1,\\ 
a_1 & \geq \frac{1}{28}(15-8\alpha-\sqrt{3}\sqrt{96\alpha^2-80\alpha+19}).
\end{align*} 
\end{lemma}
\begin{proof}
Let $\delta=a_1=\min K$ at point $p$ then we have
\begin{align*}
a_3 &=\alpha,\\
a_1 &=\delta,\\
a_2 &= 1-\alpha-\delta,\\
a_2+a_3 &=1-\delta,\\
a_2a_3 &= (1-\alpha-\delta)\alpha. 
\end{align*}

If $b_2b_3\geq 0$ then equation (\ref{Hamilton}) becomes
\[\delta=a_1\geq a_1^2+2a_2a_3\geq \delta^2+2\delta(1-2\delta)=2\delta-3\delta^2.\]
If $a_1< 0$ then $a_2< 0$ and, consequently $a_3> 1$, a contradiction. So either $a_1=a_2=0$ or $a_1=a_2=\frac{1}{3}$. \\

If $b_2b_3< 0$ then we consider two cases.

\textbf{Case 1:} $4a_2 > 1+a_1$ and, by the discussion following Lemma \ref{algebraic2}, we have 
\begin{align*}
\delta=a_1 &\geq a_1^2+2a_2a_3+b_2^2+b_3^2+4b_2 b_3\\
 \delta &\geq \delta^2+2(1-\alpha-\delta)\alpha-\frac{( 2\alpha-1+\delta)^2}{2};\\
 0 &\geq \delta^2- 8\alpha\delta+8\alpha-8\alpha^2-1.
\end{align*} 
As a consequence, 
\[4\alpha+\sqrt{24\alpha^2+1-8\alpha}\geq \delta\geq 4\alpha-\sqrt{24\alpha^2+1-8\alpha}.\]
So, 
\begin{align*}
1 &=a_1+a_2+a_3\\
 &> a_1+\frac{1}{4}(1+a_1)+\alpha,\\
\frac{3}{4}&> \frac{5}{4}(4\alpha-\sqrt{24\alpha^2+1-8\alpha})+\alpha\\
&> 6\alpha -\frac{5}{4}\sqrt{24\alpha^2+1-8\alpha}. 
\end{align*}
But that is a contradiction for $\frac{1}{3}\leq \alpha\leq 1$. Thus this case does not hold.

\textbf{Case 2:} $4a_2 \leq 1+a_1$ and, by the discussion following Lemma \ref{algebraic2}, we have 
\begin{align*}
\delta=a_1 &\geq a_1^2+2a_2a_3+b_2^2+b_3^2+4b_2 b_3\\
 \delta &\geq \delta^2+2(1-\alpha-\delta)\alpha + (1-\alpha-2\delta)^2-\frac{(\alpha-\delta)^2}{3};\\
 0 &\geq 14\delta^2+(8\alpha-15)\delta+3-4\alpha^2.
\end{align*} 
As a consequence, 
\[\frac{1}{28}(15-8\alpha+\sqrt{3}\sqrt{96\alpha^2-80\alpha+19}\geq \delta\geq \frac{1}{28}(15-8\alpha-\sqrt{3}\sqrt{96\alpha^2-80\alpha+19}.\]

\end{proof}

\begin{coro}
\label{possec}
At $p$, if $a_3\leq \frac{\sqrt{3}}{2}\approx .866025$ then $K\geq 0$.
\end{coro}

\begin{proof}
Let $a_3=\alpha$ at point p. Then by Lemma \ref{Kupper},
\begin{align*}
a_1 &\geq 0,\\
\leftarrow 15-8\alpha &\geq \sqrt{3}\sqrt{96\alpha^2-80\alpha+19} ,\\
\leftrightarrow -\frac{\sqrt{3}}{2} \leq \alpha &\leq \frac{\sqrt{3}}{2}. 
\end{align*}
\end{proof}
\begin{remark}
Z. Zhang \cite{zhang2016four} obtains a similar result but we fail to follow the proof. 
\end{remark}

\begin{lemma}
\label{Kdiffbound}
At point $p$ suppose $0\leq a_3-a_2=\gamma<2$ then 
\begin{align*}
a_1 &\geq \min\{\frac{1}{6}(3-2\gamma-\sqrt{1+8\gamma^2-4\gamma}), \frac{1}{3}(2-\sqrt{1+3\gamma^2})\}.
\end{align*}
\end{lemma}

\begin{proof}
Let $\delta=a_1=\min K$ at point $p$ then we have
\begin{align*}
a_2+a_3 &=1-\delta,\\
a_3-a_2 &= \gamma,\\
a_2-a_1 &=\frac{1-\gamma-3\delta}{2},\\
4a_2a_3=(a_2+a_3)^2-(a_2-a_3)^2 &= (1-\delta)^2-\gamma^2.
\end{align*}

If $b_2b_3\geq 0$ then equation (\ref{Hamilton}) yields
\begin{align*}
\delta=a_1 &\geq a_1^2+2a_2a_3= \delta^2+\frac{1}{2}((1-\delta)^2-\gamma^2);\\
0 &\geq \frac{3}{2}\delta^2-2\delta+\frac{1}{2}(1-\gamma^2)
\end{align*}
Therefore we have, 
\[\frac{1}{3}(2-\sqrt{1+3\gamma^2})\leq \delta \leq \frac{1}{3}(2+\sqrt{1+3\gamma^2}). \]

If $b_2b_3< 0$ then we consider two cases.

\textbf{Case 1:} If $4a_2> a_1+1$ then, by the discussion following Lemma \ref{algebraic2}, equation (\ref{Hamilton}) becomes
\begin{align*}
\delta=a_1 &\geq a_1^2+2a_2a_3+b_2^2+b_3^2+4b_2 b_3\\
 \delta &\geq \delta^2+\frac{1}{2}((1-\delta)^2-\gamma^2)-\frac{\gamma^2}{2};\\
 0 &\geq 3\delta^2- 4\delta+1-2\gamma^2.
\end{align*} 
As a consequence, 
\[\frac{2+\sqrt{1+6\gamma^2}}{3}\geq \delta\geq \frac{2-\sqrt{1+6\gamma^2}}{3}.\]
So,
\begin{align*}
1 &=a_1+a_2+a_3,\\
&=a_1+2a_2+\gamma > \frac{1}{2}+\frac{3}{2}a_1+\gamma;\\
\frac{1}{2}&> \alpha+1-\frac{1}{2}\sqrt{1+6\gamma^2}. 
\end{align*}
For $0\leq \gamma<2$, the inequality above is impossible. So this case does not happen.

\textbf{Case 2:} If $4a_2\leq a_1+1$ then, by the discussion following Lemma \ref{algebraic2}, equation (\ref{Hamilton}) becomes
\begin{align*}
\delta=a_1 &\geq a_1^2+2a_2a_3+b_2^2+b_3^2+4b_2 b_3\\
 \delta &\geq \delta^2+\frac{1}{2}((1-\delta)^2-\gamma^2)+(\frac{1-\gamma-3\delta}{2})^2-\frac{1}{3}(\frac{1+\gamma-3\delta}{2})^2;\\
 0 &\geq 3\delta^2- \delta(3-2\gamma)+\frac{2}{3}-\frac{2\gamma}{3}-\frac{\gamma^2}{3}.
\end{align*} 
As a consequence, 
\[\frac{1}{6}(3-2\gamma+\sqrt{1+8\gamma^2-4\gamma})\geq \delta\geq \frac{1}{6}(3-2\gamma-\sqrt{1+8\gamma^2-4\gamma}).\]

\end{proof}
\begin{remark}
\label{rmdiffK}
It is straightforward to check that each lower bound is a decreasing function for $\gamma\geq 0$. Moreover, for $\gamma\geq \frac{\sqrt{3}-1}{4}$,  
\[\frac{1}{3}(2-\sqrt{1+3\gamma^2})\geq \frac{1}{6}(3-2\gamma-\sqrt{1+8\gamma^2-4\gamma}).\]
\end{remark}
\begin{coro}
\label{possec1}
At $p$, if $a_3-a_2\leq \sqrt{3}-1$ then $K\geq 0$.
\end{coro}
\begin{proof}
Let $\sqrt{3}-1=\gamma$. Then by Lemma \ref{Kdiffbound} and Remark \ref{rmdiffK},
\begin{align*}
6 a_1 &\geq 3-2\gamma -\sqrt{1+8\gamma^2-4\gamma},\\
 &\geq 0.  
\end{align*}
The result follows.
\end{proof}

Lemmas \ref{Kupper} and \ref{Kdiffbound} imply that, at p, $a_2$ and $a_1$ are close to each other if there is an upper bound on $a_3$ or $a_3-a_2$. The next result estimates $a_2-a_1$ by $a_1$. 
\begin{lemma}
\label{a2a1}
Suppose at point $p$, 
\begin{align*}
a_1 &=\delta\geq 0,\\
a_2 &=x+\delta\leq \frac{1}{4}+\frac{\delta}{4}.
\end{align*}
Then \[0\leq x\leq 1-3\delta -\frac{1}{2}\sqrt{3+18\delta^2-15\delta}.\]
\end{lemma}
\begin{proof}
We have \[a_3 =\alpha=1-2\delta-x.\]
If $b_2b_3\geq 0$ then equation (\ref{Hamilton}) yields
\begin{align*}
\delta=a_1 &\geq a_1^2+2a_2a_3\geq \delta^2+2(x+\delta)(1-2\delta-x);\\
0 &\leq 2x^2 +x(6\delta-2)+3\delta^2-\delta.
\end{align*}
Therefore,
\[\delta\geq \frac{1}{6}(1-6x+\sqrt{12x^2+12x+1}).\]
So, for $0\leq x\leq 1$,
\[a_1+a_2=2a_1+x\geq x+ \frac{1}{3}(1-6x+\sqrt{12x^2+12x+1})\geq \frac{2}{3}.\]
Thus, $x=0$. \\

We consider $b_2b_3<0$, since $4a_2\leq 1+a_1$, applying the discussion after Lemma \ref{algebraic2} into equation (\ref{Hamilton}) yields
\begin{align*}
\delta =a_1 &\geq \delta^2+2(x+\delta)\alpha+(b_2^2+b_3^2+4b_2b_3),\\
\delta =a_1 &\geq \delta^2+2(x+\delta)\alpha+\frac{1}{3}(2x^2-2x(\alpha-x-\delta)-(\alpha-x-\delta)^2),\\
&\geq \delta^2+2(x+\delta)\alpha+x^2-\frac{1}{3}(\alpha-\delta)^2;\\
0 &\geq x^2+2x\alpha+2\delta\alpha+\delta^2-\delta-\frac{1}{3}(\alpha-\delta)^2.
\end{align*} 
Substituting $\alpha=1-2\delta-x$ then yields,
\begin{align*}
0 &\leq 4x^2-8x(1-3\delta)+(1-9\delta+18\delta^2). 
\end{align*}
This quadratic has the discriminant
\[D=48(1-5\delta+6\delta^2)=48(1-2\delta)(1-3\delta)\geq 0.\]
As a consequence,
\[ x\leq 1-3\delta -\frac{1}{2}\sqrt{3+18\delta^2-15\delta}.\]
Then the result follows. 
\end{proof}
\begin{remark}
For $0\leq \delta\leq \frac{1}{3}$, 
\[1-3\delta -\frac{1}{2}\sqrt{3+18\delta^2-15\delta}<\frac{1}{4}(1-3\delta).\]
\end{remark}
\section{Classification}
\label{classify}
This section proves the main result. 
\begin{theorem}
\label{mainupper}
Let $(M, g)$ be an Einstein four-manifold $\Rc=g$ 
such that at each point, \[K\leq \beta=\frac{14-\sqrt{19}}{12}\approx .8034.\]
Then it is isometric to either $(S^4,g_{0})$, $(\mathbb{RP}^4, g_0)$, or $(\mathbb{CP}^2,\ g_{FS})$ up to rescaling. 
\end{theorem}
\begin{proof}
We consider Berger's decomposition Prop \ref{berger} at point $p$ realizing the minimum of sectional curvature. Let $\alpha=a_3$. By Lemma \ref{Kupper},
\[a_1\geq \alpha_1=\frac{1}{28}(15-8\alpha-\sqrt{3}\sqrt{96\alpha^2-80\alpha+19}).\]

Since this function of $\alpha$ is decreasing for $\alpha\geq \frac{1}{3}$, we conclude that,
\[a_1\geq \beta_1=\frac{1}{28}(15-8\beta-\sqrt{3}\sqrt{96\beta^2-80\beta+19}).\]

By the choice of $p$, the sectional curvature is at least $\beta_1$ at each point. 

By Lemma \ref{k3k1}, then at each point 
\begin{align*}
|W^+|+|W^-|&\leq \frac{6(K_3+K_2)+2(K_3-K_2)-4}{\sqrt{6}},\\
&\leq \frac{8K_3+4K_2-4}{\sqrt{6}}=\frac{4K_3-4K_1}{\sqrt{6}},\\
&\leq \frac{4\beta-4\beta_1}{\sqrt{6}}.
\end{align*}

Substituting $\beta=\frac{14-\sqrt{19}}{12}$ yields,
\[|W^+|+|W^-|\leq \frac{3}{\sqrt{6}}=\sqrt{\frac{3}{2}}.\]

The theorem then follows from Prop \ref{propWpm}. 
\end{proof}

Similarly, we have the following.
\begin{theorem}
\label{maindiff}
Let $(M, g)$ be an Einstein four-manifold $\Rc=g$ 
such that at each point, \[a_3-a_2\leq \beta=\frac{7-\sqrt{19}}{4}\approx .660275.\]
Then it is isometric to either $(S^4,g_{0})$, $(\mathbb{RP}^4, g_0)$, or $(\mathbb{CP}^2,\ g_{FS})$ up to rescaling. 
\end{theorem}
\begin{proof}
We consider Berger's decomposition Prop \ref{berger} at point $p$ realizing the minimum of sectional curvature. Let $\gamma=a_3-a_2$. By Lemma \ref{Kdiffbound},
\[a_1\geq \min\{\frac{1}{6}(3-2\gamma-\sqrt{1+8\gamma^2-4\gamma}), \frac{1}{3}(2-\sqrt{1+3\gamma^2})\}.\]

Notice that each lower bound above is a decreasing function of $\gamma$, and the maximum value of $\gamma$ is greater than  $\frac{\sqrt{3}-1}{4}$, by Remark \ref{rmdiffK}, we conclude that,
\[a_1\geq \beta_1=\frac{1}{6}(3-2\beta-\sqrt{1+8\beta^2-4\beta}).\]

By the choice of $p$, the sectional curvature is at least $\beta_1$ at each point. 

By Lemma \ref{k3k1}, then at each point 
\begin{align*}
|W^+|+|W^-|&\leq \frac{6(a_3+a_2)+2(a_3-a_2)-4}{\sqrt{6}},\\
&\leq \frac{2-6a_1+2(a_3-a_2)}{\sqrt{6}}\\
&\leq \frac{2+2\beta-6\beta_1}{\sqrt{6}}.
\end{align*}
Substituting $\beta=\frac{7-\sqrt{19}}{4}$ yields,
\[|W^+|+|W^-|\leq \frac{3}{\sqrt{6}}=\sqrt{\frac{3}{2}}.\]

The theorem then follows from Prop \ref{propWpm}. 
\end{proof}

Now the proof of Theorem \ref{maineinstein} follows immediately.
\begin{proof}
Condition (a.) is considered in Theorem \ref{mainupper}.

Condition (b.) implies that 
\begin{align*}
2a_2+a_1 &\geq \frac{\sqrt{19}-3}{4},\\
a_2-a_3 &\geq \frac{\sqrt{19}-7}{4}.
\end{align*}
This is exactly the condition of Theorem \ref{maindiff} so the result follows. 
\end{proof}

\def\cprime{$'$}
\bibliographystyle{plain}
\bibliography{bioEin}

\def\cprime{$'$}
\begin{thebibliography}{10}

\bibitem{berger61}
Marcel Berger.
\newblock Sur quelques vari\'et\'es d'{E}instein compactes.
\newblock {\em Ann. Mat. Pura Appl. (4)}, 53:89--95, 1961.

\bibitem{Besse}
Arthur~L. Besse.
\newblock {\em Einstein manifolds}, volume~10 of {\em Ergebnisse der Mathematik
  und ihrer Grenzgebiete (3) [Results in Mathematics and Related Areas (3)]}.
\newblock Springer-Verlag, Berlin, 1987.

\bibitem{branson00}
T.~Branson.
\newblock Kato constants in {R}iemannian geometry.
\newblock {\em Math. Res. Lett.}, 7(2-3):245--261, 2000.

\bibitem{brendle10einstein}
Simon Brendle.
\newblock Einstein manifolds with nonnegative isotropic curvature are locally
  symmetric.
\newblock {\em Duke Math. J.}, 151(1):1--21, 2010.

\bibitem{CGH00weight}
David M.~J. Calderbank, Paul Gauduchon, and Marc Herzlich.
\newblock Refined {K}ato inequalities and conformal weights in {R}iemannian
  geometry.
\newblock {\em J. Funct. Anal.}, 173(1):214--255, 2000.

\bibitem{caotran1}
Xiaodong Cao and Hung Tran.
\newblock The {W}eyl tensor of gradient {R}icci solitons.
\newblock {\em Geom. Topol.}, 20(1):389--436, 2016.

\bibitem{cao14einstein}
Xiaodong Cao and Peng Wu.
\newblock Einstein four-manifolds of three-nonnegative curvature operator.
\newblock {\em Unpublished}, 2014.

\bibitem{costa04}
{\'E}zio de~Araujo~Costa.
\newblock On {E}instein four-manifolds.
\newblock {\em J. Geom. Phys.}, 51(2):244--255, 2004.

\bibitem{Derd83}
Andrzej Derdzi{\'n}ski.
\newblock Self-dual {K}\"ahler manifolds and {E}instein manifolds of dimension
  four.
\newblock {\em Compositio Math.}, 49(3):405--433, 1983.

\bibitem{gl99}
Matthew~J. Gursky and Claude Le{B}run.
\newblock On {E}instein manifolds of positive sectional curvature.
\newblock {\em Ann. Global Anal. Geom.}, 17(4):315--328, 1999.

\bibitem{H3}
Richard~S. Hamilton.
\newblock Three-manifolds with positive {R}icci curvature.
\newblock {\em J. Differential Geom.}, 17(2):255--306, 1982.

\bibitem{Lich58}
Andr{\'e} Lichnerowicz.
\newblock {\em G\'eom\'etrie des groupes de transformations}.
\newblock Travaux et Recherches Math\'ematiques, III. Dunod, Paris, 1958.

\bibitem{myer41}
S.~B. Myers.
\newblock Riemannian manifolds with positive mean curvature.
\newblock {\em Duke Math. J.}, 8:401--404, 1941.

\bibitem{st69}
I.~M. Singer and J.~A. Thorpe.
\newblock The curvature of {$4$}-dimensional {E}instein spaces.
\newblock In {\em Global {A}nalysis ({P}apers in {H}onor of {K}. {K}odaira)},
  pages 355--365. Univ. Tokyo Press, Tokyo, 1969.

\bibitem{tachibana}
Shun-ichi Tachibana.
\newblock A theorem of {R}iemannian manifolds of positive curvature operator.
\newblock {\em Proc. Japan Acad.}, 50:301--302, 1974.

\bibitem{tran16weyl}
Hung Tran.
\newblock On closed manifolds with harmonic {W}eyl curvature.
\newblock {\em Adv. Math.}, 322:861--891, 2017.

\bibitem{yangdg00}
DaGang Yang.
\newblock Rigidity of {E}instein {$4$}-manifolds with positive curvature.
\newblock {\em Invent. Math.}, 142(2):435--450, 2000.

\bibitem{zhang2016four}
Zhuhong Zhang.
\newblock Four-dimensional einstein manifolds with sectional curvature bounded
  from above.
\newblock {\em arXiv preprint arXiv:1606.01157}, 2016.

\end{thebibliography}

\end{document}